\documentclass[12pt]{article} 

\usepackage{amsmath} 
\usepackage{amsthm}
\usepackage{amssymb}
\usepackage{mathtools}
\usepackage[hidelinks]{hyperref}
\usepackage[utf8]{inputenc}
\usepackage{csquotes}
\usepackage[pagewise]{lineno}
\usepackage[english]{babel}
\usepackage{enumitem}
\usepackage{titlesec}
\usepackage{bm}
\usepackage[title]{appendix}
\usepackage{tikz-cd}

\usepackage{geometry}
\makeatletter
\BeforeBeginEnvironment{proof}{%
  \def\@Opargbegintheorem#1#2#3#4{#4\trivlist
      \item[]{#3#2\@thmcounterend\ }}%
}
\AfterEndEnvironment{proof}{%
  \def\@Opargbegintheorem#1#2#3#4{#4\trivlist
      \item[\hskip\labelsep{#3#1}]{#3(#2)\@thmcounterend\ }}%
}
 \newtheorem{thm}{Theorem}[subsection]
 \newtheorem{cor}[thm]{Corollary}
 \newtheorem{lem}[thm]{Lemma}
 \newtheorem{defn}{Definition}[subsection]
 \newtheorem{exmp}{Example}[subsection]
 \newtheorem{rema}{Remark}[subsection]

\def\rem#1#2{ #1\;\textnormal{rem}\;#2}

\def\res#1#2{ \textnormal{eval}\left(#1;\  #2\right)}
\providecommand{\keywords}[1]
{
	\small	
	\textbf{\text{Keywords:}} #1
}

\providecommand{\subclass}[1]
{
	\small	
	\textbf{\text{MSC(2010):}} #1
}

\title{\Large The Restricted Partition and $q$-Partial Fractions \thanks{Dedicated to Bhagawan Sri Sathya Sai Baba.}}

\author{\normalsize N. Uday Kiran\thanks{nudaykiran@sssihl.edu.in}  \\
	\small Department of Mathematics and Computer Science\\
	\small Sri Sathya Sai Institute of Higher Learning, Puttaparthi, India \\
}

\date{}

\begin{document}
    \maketitle
	\begin{abstract}

         The restricted partition function $p_{N}(n)$ counts the partitions of $n$ into at most $N$ parts. In the nineteenth century Sylvester showed that these partitions can be expressed as a sum of $k$-periodic quasi-polynomials ($1\leq k\leq N$) which he termed as Waves. It is now well-known that one can easily perform a wave decomposition using a special type of partial fraction decomposition (the so-called $q$-partial fractions) of the generating function of $p_{N}(n)$. In this paper we show that the coefficients of these $q$-partial fractions can be expressed as a linear combination of the Ramanujan sums. In particular, we show, for the first time, an appearance of the degenerate Bernoulli numbers, the degenerate Euler numbers and a special generalization of the Ramanujan sums, which we term as a Gaussian-Ramanujan sum, in the formulae for certain waves. These coefficients not only provide a good approximation of $p_{N}(n)$ but they can also be used for obtaining good bounds. Further, we provide a combinatorial meaning to these sums. Our approach for partial fractions is based on a projection operator on the $I$-adic completion of the ring of polynomials, where $I$ is an ideal generated by the Cyclotomic polynomial. \\
		\\
		\keywords{ Restricted Partition $\cdot$ Sylvester Waves $\cdot$ Cyclotomic Polynomials $\cdot$ Ramanujan Sum  $\cdot$ Degenerate Numbers  }\\
		\subclass{05A15 $\cdot$  11P82 $\cdot$ 11B68}
	\end{abstract}


\section{Introduction}
Let the restricted $q$-product be $(x)_{N}=\prod^{N}_{i=1}(1-x^{i})$. Then, the generating function of the restricted partition of $n$ into parts none of which exceed $N$ is given by $$F_{N}(x)=1/(x)_{N}=\sum_{n=0}^{\infty}p_{N}(n)x^{n}.$$ Sylvester made the following remarkable decomposition:  
$$p_{N}(n)=\sum_{k=1}^{N} W_{k}(n;N).$$ The term $W_{k}(n;N)$, called a $k^{th}$ Sylvester wave, a $k$-periodic quasi-polynomial in $n$ given by the coefficient of $z^{-1}$, i.e. the residue at $z=0$ of the function 
\begin{equation*}
W_{j}(n;N)=\textbf{Res}_{z=0}\sum_{\xi_{j}}\frac{\xi_{j}^{n}e^{nz}}{(1-\xi_{j}^{-1}e^{-z})(1-\xi_{j}^{-2}e^{-2z})\cdots (1-\xi_{j}^{-N}e^{-Nz})},
\end{equation*}
where $\xi_{j}=e^{2\pi i/j}$. Moreover, $\xi_{1}=1$ so that $W_{1}(t;\textbf{A})$ is a polynomial in $t$.  

A $q$-partial fraction \cite{munagi2, uk} of $F_{N}(x)$ can be written with rational coefficients $\Gamma_{hkl}(N)$ as
\begin{eqnarray}
\nonumber F_{N}(x)&=&\sum_{l=1}^{N}\frac{g_{1,l}^{(N)}(x)}{(1-x)^{l}}+\sum_{l=1}^{\lfloor N/2\rfloor}\frac{g_{2,l}^{(N)}(x)}{(1+x)^{l}}+\sum_{k=3}^{N}\sum_{l=1}^{\lfloor N/k\rfloor}\frac{g_{kl}^{(N)}(x)}{(1-x^{k})^{l}} \\
&=& \sum_{l=1}^{N}\frac{\Gamma_{0,1,l}(N)}{(1-x)^{l}}+\sum_{l=1}^{\lfloor N/2\rfloor}\frac{\Gamma_{0,2,l}(N)}{(1+x)^{l}}+\sum_{k=3}^{N}\sum_{l=1}^{\lfloor N/k\rfloor}\sum_{h=0}^{k-1}\frac{\Gamma_{hkl}(N)x^{h}}{(1-x^{k})^{l}}.\label{qpf}
\end{eqnarray}
Then, one can directly express the $1^{st}$ and the $2^{nd}$ wave as
\begin{equation}\label{12_wave}
W_{1}(n;N)=\sum_{l=1}^{N}\begin{pmatrix}
n+l-1 \\ n
\end{pmatrix} \Gamma_{0,1,l}\quad  \textnormal{ and}\quad 
W_{2}(n;N)=(-1)^{n}\sum_{l=1}^{\lfloor N/2\rfloor}\begin{pmatrix}
n+l-1 \\ n
\end{pmatrix}\Gamma_{0,2,l}.
\end{equation}
Using $\Gamma_{hkl}(N)$ the $k^{th}$ wave $W_{k}(n;N)$ can be written as 
$$
W_{k}(n; N)=\sum^{\lfloor N/k\rfloor}_{l=1}\sum_{i=0}^{l-1}\begin{pmatrix}\lfloor \frac{n}{k}\rfloor +l-i-1\\ l-i-1\end{pmatrix}\Gamma_{(n\%k)kl}(N),$$
where $\lfloor t \rfloor$ is the greatest integer $\leq t$ and $\%$ is the remainder operator. Furthermore, if one has a $q$-partial fraction it is much easier to use a circulator notation for periodic sequence, which occur in the restricted partition formulae.

Surprisingly, in spite of the utility of the coefficients $\Gamma_{hkl}(N)$, not much work on the arbitrary $N$ case is available in the literature (see for instance \cite{munagi2, sullivian, Rubin}). In this work, using an algebraic formalism developed in \cite{uk} we address the arbitrary $N$ case. Our main contribution in this work is the observation of an appearance of the special trigonometric sums (degenerate Bernoulli numbers, degenerate Euler numbers and a generalization of the Ramanujan sums) in the coefficients on the Sylvester waves. We also introduce the term Gaussian-Ramanujan sum for the generalization of the Ramanujan sum and briefly study its properties.

Glaiser provided direct formulae for $W_{1}(n;N)$ \cite{glaisher} using the Bernoulli numbers. O' Sullivian \cite{sullivian} showed that for $n \leq 1480$ the first waves  is a good approximation to $p(n)$, the unrestricted partition function. Rubinstein et.al. \cite{Rubin} showed that certain higher-order Bernoulli number can be used to express the waves. Moreover, Rubinstein et.al. show that the waves are not only periodic but also satisfy the same linear recurrence relation satisfied by $p_{N}(n)$. In this paper, we develop formulae based on trigonometric sums that are amenable to good estimations.

\section{Main Results} 

In this section, we briefly state the methodology and summarize the main results of the paper. The first result is on a formula for $g_{kl}^{(N)}(x)$ given in (\ref{qpf}) and the other three results are on the coefficient $\Gamma_{hkl}(N)$ for specific indices $k$ and $l$. 

The tool central to our investigation is the polynomial-valued \textit{eval} operator. For $r(x),s(x),a(x)$ non-trivial polynomials with $\alpha(x) s(x)=1 \textnormal{ mod }a(x)$ we define 
$$
\res{\frac{r(x)}{s(x)}}{a(x)} = \rem{(\alpha(x)r(x))}{a(x)},
$$ 
where rem is the polynomial remainder operator. As we show in Section \ref{sec_eval}, the eval operator can be treated as a projection operator on the $I$-adic completion, where $I$ is the ideal generated by $a(x)$. 

In \cite{uk}, the current author had demonstrated the effectiveness of the eval operator for partial fraction decomposition with the so-called extended cover-up method. Our $q$-partial fraction decomposition of $F_{N}(x)$ relies on factorizing the denominator into cyclotomic polynomials. To recall, the $n^{th}$ cyclotomic polynomial is a monic and irreducible polynomial in $\mathbb{Z}[x]$ denoted $\Phi_{n}(x)$ which satisfies  
\begin{equation}\label{factor_phi}
1-x^{n}=\prod_{d|n} \Phi_{d}(x)
\end{equation} 
and $\Phi_{1}(x)=1-x$. Furthermore, distinct cyclotomic polynomials are pairwise relatively prime, i.e. for $m\neq n$ we have $\textnormal{gcd}(\Phi_{m}(x),\Phi_{n}(x))=1$. In view of these observations one can easily factorize $F_{N}(x)$ into irreducible factors
\begin{equation}\label{factor_FNx}
F_{N}(x)=\frac{1}{\prod_{k=1}^{N}(1-x^k)}=\frac{1}{\prod_{k=1}^{N}\Phi_{k}(x)^{\lfloor N/k\rfloor}},
\end{equation}
where $\lfloor t\rfloor$ is the greatest integers $\leq t$. Equipped with the extended cover-up method (stated in Theorem \ref{cover_up} and proved in \cite{uk}) one can easily obtain a $q$-partial fraction 
\begin{equation}\label{qPF1}
\frac{1}{\prod_{k=1}^{N}(1-x^k)}=\sum_{k=1}^{N}\frac{h_{k}^{(N)}(x)}{\Phi_{k}(x)^{\lfloor N/k\rfloor}}=\sum_{k=1}^{N}\frac{\Theta_{k}(x)^{\lfloor N/k\rfloor}h_{k}^{(N)}(x)}{(1-x^k)^{\lfloor N/k\rfloor}},
\end{equation}
where $\Theta_{k}(x)$ satisfies $\Theta_{k}(x)\Phi_{k}(x)=1-x^k$ and is called the inverse cyclotomic polynomial \cite{Moree}. Here $h_{k}^{(N)}(x)$ is given by 
\begin{equation}\label{hNk}
h_{k}^{(N)}(x)=\textnormal{eval}\left(\frac{1}{\Phi_{1}(x)^{N}\cdots \widehat{\Phi_{k}(x)^{\lfloor N/k\rfloor}} \cdots \Phi_{N}(x)};\quad\Phi_{k}(x)^{\lfloor N/k\rfloor}\right),
\end{equation}
where $\widehat{}$ means dropping the corresponding term. 

Our first main result in this paper is Theorem \ref{main_qpf} that simplifies $h_{k}^{(N)}(x)$ in (\ref{qPF1}) to the term $g_{kl}^{(N)}(x)$ in (\ref{qpf}). Subsequently, for specific choice of $k$ and $l$ we determine the structure of the polynomial $g_{kl}^{(N)}(x)$. Recasting a result in \cite{uk} we obtain the following coefficients in terms of degenerate Bernoulli numbers. 

\begin{thm}[\cite{uk}] 
For $1\leq j< N$ we have 
\begin{equation}\label{1_wave}
\Gamma_{01(N-j)}(N) = \frac{(-1)^{j}}{N!}\sum_{j_{2}+\cdots+j_{N}=j}\frac{\tilde{\beta}_{j_{2}}(2)\cdots \tilde{\beta}_{j_{N}}(N)}{j_{2}!\cdots j_{N}!}    
\end{equation}
and $\Gamma_{01N}(N)=1/N!$. Here $\tilde{\beta}_{j}(n)=n^{k}\beta_{k}(1/n)$, where $\beta_{k}(\lambda)$ is the degenerate Bernoulli number. These terms are explained in Section \ref{sec_W1W2}. 
\end{thm}

It is a pleasant surprise that the coefficients $\Gamma_{20l}(N)$ of second wave is similar to (\ref{1_wave}), but with alternating degenerate Bernoulli and degenerate Euler numbers. We state the result below and prove it in Section

\begin{thm}\label{main_W2}
For $1\leq j< \lfloor N/2\rfloor$ we have
\begin{equation}\label{2_wave}
\Gamma_{02(\lfloor N/2\rfloor-j)}(N) =\frac{(-1)^{j}}{2^{N}\lfloor N/2\rfloor!} \sum_{j_{1}+\cdots+j_{N}=j}\frac{\tilde{\gamma}_{j_{1}}(1)\cdots\tilde{\gamma}_{j_{N}}(N)}{j_{1}!\cdots j_{N}!}    
\end{equation}
and $\Gamma_{02\lfloor N/2\rfloor}=\frac{1}{2^{N}\lfloor N/2\rfloor!}$ and $g_{2l}^{(N)}(x)=\Gamma_{02l}(N)$ in (\ref{qpf}) where
$$  
\tilde{\gamma}_{k}(n)=\begin{cases}
\tilde{\varepsilon}_{k}(n)=n^{k}\varepsilon_{k}(1/n) & \textnormal{ for } n \textnormal{ odd},\\
\tilde{\beta}_{k}(n)=n^{k}\beta_{k}(1/n) & \textnormal{ for } n \textnormal{ even},
\end{cases}
$$
where $\beta_{k}(\lambda)$ and $\varepsilon_{k}(\lambda)$ are the degenerate Bernoulli number and degenerate Euler number respectively. These terms are discussed in Section \ref{sec_W1W2}. 
\end{thm}

By the formula (\ref{2_wave}) we can completely characterize the second wave $W_{2}(n;N)$ using (\ref{12_wave}). Clearly, due to the factor $(-1)^{n}$ we can easily see that $W_{2}(n;N)$ is a quasi-polynomial in $n$. Also interestingly, we observe that except for the $(-1)^{n}$ factor the term $W_{2}(n;N)$ is a polynomial in $n$. 

Further, from (\ref{euler_trig}), (\ref{euler_odd}) and \cite{uk} we have degenerate numbers expressed using trigonometric sums:
$$
\tilde{\varepsilon}_{k}(m)=k!\frac{(-1)^k}{m}\sum_{j=0}^{m-1}\frac{2\xi^{j}}{(1+\xi^{j})^{k+1}}\quad \textnormal{ and }\quad\tilde{\beta}_{k}(m)=k!\frac{(-1)^{k+1}}{m}\sum_{j=1}^{m-1}\frac{\eta^{j}}{(1-\eta^{j})^{k+1}},
$$
where $\xi^{m}+1=0$ and $\eta^{m}-1=0$; further, $\xi^{j}+1=0$ and $\eta^{j}-1=0$ for $0< j<m$. 

Our next main result relies on a Fourier Analysis on $g_{k\lfloor N/k \rfloor}(x)$ in order to obtain an exact form of its coefficients $\Gamma_{jk\lfloor N/k\rfloor}$. With these formulae we can completely characterize and provide a direct formula for the top-order term of $W_{k}(n;N)$. It is well known that the top-order terms approximate the waves well \cite{sullivian}.

\begin{thm}\label{main_gamma_jkNk}
The coefficients of $g_{k\lfloor N/k\rfloor}(x)$ are given by 
$$
\Gamma_{jk\lfloor N/k\rfloor}(N)=\frac{1}{k^{\lfloor N/k\rfloor+1}\lfloor N/k\rfloor!}\sigma_{k}(j;N\%k),
$$
where we call the term $\sigma_{k}(\cdot,\cdot)$ Gaussian-Ramanujan Sum (Definition \ref{GRS}) and Denote $R_{k,s}(x)=\sum_{j=0}^{k-1}\sigma_{k}(t;s)x^{t}$. Then we have
\begin{equation}\label{Rksx}
g_{k\lfloor N/k\rfloor}(x)=\frac{1}{k^{\lfloor N/k\rfloor+1}\lfloor N/k\rfloor!}R_{k,N\%k}(x).
\end{equation}
\end{thm}

We call $\sigma_{k}(t;s)$ as the Gaussian-Ramanujan sum as they are restricted $q$-product based generalizations of the Ramanujan sum. In fact, $\sigma_{k}(t;0)=c_{k}(t)$ and $\sigma_{k}(t;k-1)=(1/k)c_{k}(t)$. We show that the trigonometric sum $\sigma_{k}(t;j)$ for a fixed $k, 0\leq j<k$ and $0\leq t < k$, is an of the form $a/k$ where $a$ is an integer. We show that $\sigma_{k}(t;s)$ satisfies a linear recurrence relation which helps on to obtain efficient formulae for the associated polynomial $R_{k,s}(x)$ given in (\ref{Rksx}). A study of the polynomial $R_{k,0}(x)$, corresponding to the Ramanujan sum, is done in \cite{toth}. We show that the coefficients of $R_{k,s}(x)$ are bounded for certain $s$ but, in general, grow exponentially in $k$.

\section{The eval operator}\label{sec_eval}

The key to our methodology is the polynomial-valued eval operator. In \cite{uk} we defined the operator and demonstrated its efficacy to perform partial fractions. As the current work also relies heavily on eval, we first briefly summarize its properties and prove some more basic results necessary for this work.  

Let $r(x),s(x)$ and $p(x)$ be in the ring of polynomials $\mathbb{Q}[x]$ such that $s(x)$ and $p(x)$ are relatively prime. By B\'{e}zout's identity there exist two polynomials $\alpha(x),\beta(x)\in \mathbb{Q}[x]$ such that 
\begin{equation}\label{a_alpha}
\alpha(x)s(x)+\beta(x)p(x)=1.
\end{equation}

\begin{defn}\label{eval}
Given non-constant polynomials $r(x), s(x), p(x)\in \mathbb{Q}[x]$ and $p(x),s(x)$ satisfy (\ref{a_alpha}), we define the evaluation of the rational polynomial $\frac{r(x)}{s(x)}$ modulo $p(x)$ as
$$
\res{\frac{r(x)}{s(x)}}{p(x)} = \rem{(\alpha(x)r(x))}{p(x)},
$$ 
where $\alpha(x) s(x)=1 \textnormal{ mod }p(x)$ and rem is the polynomial remainder operator. 
\end{defn}

So, essentially the eval operator takes in a rational polynomial and gives out a polynomial with degree less than $\textnormal{deg}(p)$. This operator is well behaved under polynomial arithmetic \cite{uk}.  

\begin{lem} \label{lem_res}
Given $p(x), r_{i}(x), s_{i}(x)\in \mathbb{Q}[x]$ and $\textnormal{gcd}(s_{i}(x),p(x))=1$ for $i=0,1$. The following are some properties of $\textnormal{eval}$ function:
\begin{enumerate}
\item $\res{r_{0}(x)}{p(x)}=\rem{r_{0}(x)}{p(x)}$
\item $\textnormal{eval}\left(\frac{r_{0}(x)}{s_{0}(x)}+\frac{r_{1}(x)}{s_{1}(x)};\ p(x)\right)=\textnormal{eval}\left(\frac{r_{0}(x)}{s_{0}(x)}; p(x)\right)+\textnormal{eval}\left(\frac{r_{1}(x)}{s_{1}(x)}); p(x)\right)$
\item $\res{\frac{r_{0}(x)r_{1}(x)}{s_{0}(x)s_{1}(x)}}{p(x)}=\rem{\left\{\res{\frac{r_{0}(x)}{s_{0}(x)}}{p(x)}\res{\\\frac{r_{1}(x)}{s_{1}(x)}}{p(x)}\right\}}{p(x)}.$
\end{enumerate}
\end{lem}

The practical utility of eval relies on the following two results.   

\begin{lem}[Substitution Rule 1, \cite{uk}]
$$
\res{\frac{r_{0}(x) - p_{0}(x)p(x)}{s_{0}(x)- q_{0}(x) p(x)}}{p(x)}=\res{\frac{r_{0}(x)}{s_{0}(x)}}{p(x)},
$$
for some $p_{0}(x),q_{0}(x)\in \mathbb{Q}[x]$. Stated in words: one can successively substitute the occurrences of $p(x)$ by $0$ in both numerator and denominator of the rational polynomial given in the first argument without changing the result. 
\end{lem}

\begin{thm}[Extended Cover-Up Method, \cite{uk}]\label{cover_up}
Let $p_{1}(x),\cdots, p_{n}(x)\in \mathbb{Q}[x]$ be pairwise relatively prime polynomials and let $f(x)\in \mathbb{Q}[x]$ satisfying $deg(f)< deg(p_{1}\cdots p_{k})$. Then, we have the following identity 
\begin{equation}\label{extended_cover_up_eqn}
f(x)\prod_{j=1}^{k}\frac{1}{p_{j}(x)}=\sum^{k}_{j=1}\frac{\textnormal{eval}(f(x)/g_{j}(x);p_{j}(x))}{p_{j}(x)},
\end{equation}
where $g_{j}(x)=p_{1}(x)\cdots\widehat{p_{j}(x)}\cdots p_{k}(x)$ and $\ \widehat{ }$ refers to dropping the corresponding factor.  
\end{thm}

It is easier to visualize the mechanism of extended cover-up method for a smaller case. The partial fractions (\ref{extended_cover_up_eqn}) can be viewed as   
$$
\frac{f}{p_{1}p_{2}p_{3}} = \frac{\textnormal{eval}\left(\frac{f}{\quad p_{2}p_{3}};p_{1}\right)}{p_{1}}+\frac{\textnormal{eval}\left(\frac{f}{p_{1}\quad p_{3}};p_{2}\right)}{p_{2}}+\frac{\textnormal{eval}\left(\frac{f}{p_{1}p_{2}\quad};p_{3}\right)}{p_{3}}.
$$

From the above expression one can easily notice that in order to obtain the term associated with $p_{i}(x)$ one has to drop $p_{i}(x)$ factor from the denominator (as if covering it) and evaluate the resulting expression with $p_{i}(x)$. 

For our further computations we need to perform eval with respect to powers of irreducible polynomials. 

\begin{thm}[$I$-adic expansion]
Let $p(x)\in \mathbb{Q}[x]$ be an irreduciable polynomial with $\textnormal{deg}(p)\geq 1$ and $\langle p\rangle$ be the corresponding maximal ideal. Then, the following holds:
For $a\in \widehat{\mathbb{Q}[x]}^{\langle p\rangle},$ there exists the unique $I$-adic expansion in $\widehat{\mathbb{Q}[x]}^{\langle p\rangle}$ such that 
$$
a=(a_{0},a_{0}+a_{1}p,\cdots,\sum_{i=0}^{k-1}a_{i}p^{i},\cdots)\in \widehat{\mathbb{Q}[x]}^{\langle p\rangle},
$$
where $a_{i}\in \mathbb{Q}[x]$ is a polynomial with $\textnormal{deg}(a_{i})<\textnormal{deg}(p)$. Furthermore, $\textnormal{eval}(\cdot; p(x)^{k}):\mathbb{Q}[x]_{\langle p\rangle}\rightarrow \mathbb{Q}[x]$ operator factors through $\widehat{\mathbb{Q}[x]}^{\langle p\rangle}$ and given by 
$$
\begin{tikzcd}
\textnormal{eval}(\ \cdot\ ;p(x)^k):\mathbb{Q}[x]_{\langle p\rangle} \arrow[r]\; \arrow[dr, dashrightarrow]
& \widehat{\mathbb{Q}[x]}^{\langle p\rangle} \arrow[d,"\pi_k"]\\
& \mathbb{Q}[x]
\end{tikzcd}
$$
where $\pi_{k}$ is a projection onto the $k^{th}$ component. 
\end{thm}

The above result can be easily seen by the below lemma.

\begin{lem}\label{eval_power_k_case}
Let $p(x)\in \mathbb{Q}[x]$ be an irreducble polynomial and $f(x)\in \mathbb{Q}[x]$ be a relatively prime to $p(x)$ such that the following equation holds 
\begin{equation}\label{gcd_f}
a(x)f(x)+b(x)p(x)=1,
\end{equation}
for some $a(x),b(x)$ two polynomials with rational coefficients. Then, we have 
$$
\textnormal{eval}\left(\frac{1}{f(x)}; p(x)^{k}\right)=\sum_{j=0}^{k-1}a_{j}(x)p(x)^{j},
$$
where $a_{j}(x)=(b(x)a^{(j)}(x))\textnormal{ rem } p(x)$ and $a^{(j)}(x)=a(x)^{j}\textnormal{ rem }f(x)$.
\end{lem}
\begin{proof}[Proof.]
The idea is to apply the extended cover-up Theorem \ref{cover_up} in a reverse direction. That is, if we have a partial fraction 
\begin{equation}\label{pf_k_case}
\frac{1}{p(x)^{k}f(x)}=\frac{b_{1}(x)}{p(x)^{k}}+\frac{b_{2}(x)}{f(x)}.
\end{equation}
Then, $\textnormal{eval}(1/f(x); p(x)^{k})= b_{1}(x) \textnormal{ rem }p(x)^{k}$. Therefore, we set out to obtain a partial fraction of the form (\ref{pf_k_case}). By (\ref{gcd_f}) we have 
\begin{equation}\label{k1case}
\frac{1}{p(x)f(x)}=\frac{a(x) \textnormal{ rem }p(x)}{p(x)}+\frac{b(x)\textnormal{ rem }f(x)}{f(x)}.
\end{equation}
We prove by induction on $k$. The case $k=1$ holds by (\ref{k1case}). Suppose the result holds for $k=n$, that is
$$
\frac{1}{p(x)^{n}f(x)}=\frac{\sum_{j=0}^{n-1}a_{j}(x)p(x)^{j}}{p(x)^{n}}+\frac{a^{(n)}(x)}{f(x)},
$$
where $a_{j}(x)=b(x)a(x)^{j}\textnormal{ rem } p(x)$ and $a^{(n)}(x)=a(x)^{n}\textnormal{ rem }f(x).$ Dividing both sides by $p(x)$ we get 
$$
\frac{1}{p(x)^{n+1}f(x)}=\frac{\sum_{j=0}^{n-1}a_{j}(x)p(x)^{j}}{p(x)^{n+1}}+\frac{a^{(n)}(x)}{p(x)f(x)}.
$$
Performing partial fraction decomposition on the second term using the extended cover-up method we have 
\begin{eqnarray}
\nonumber \frac{a^{(n)}(x)}{p(x)f(x)}&=&\frac{a^{(n)}(x)a(x)\textnormal{ rem }p(x)}{p(x)}+\frac{a^{(n)}(x)a(x)\textnormal{ rem } f(x)}{f(x)}\\
\nonumber &=& \frac{a_{n+1}(x)}{p(x)}+\frac{a^{(n+1)}(x)}{f(x)},
\end{eqnarray}
where $a_{n+1}(x)=\sum_{j=0}^{n}a_{j}(x)p(x)^{j}$ and $a^{(n+1)}(x)=a(x)^{n+1}\textnormal{ rem }f(x).$ 
Hence the result is proved.  
\end{proof}

\begin{cor}[Eval-Taylor Series]\label{eval_taylor}
Let $p(x)=(x-a)$ and let the Taylor series of $1/f(x)$ centered about $x=a$, with $f(a)\neq 0$, be given by 
\begin{equation}\label{an_coeff}
\frac{1}{f(x)}=\sum_{n=0}^{\infty} a_{n}(x-a)^{n}.
\end{equation}
Then, 
$$
\textnormal{eval}\left(\frac{1}{f_{2}(x)};(x-a)^{k}\right)=\sum_{n=0}^{k-1}a_{n}(x-a)^{n}.
$$
\end{cor}

In case we have a product of multiple polynomials $f_{1}(x),\cdots,f_{m}(x)$ then by Lemma \ref{lem_res}(3) and Theorem \ref{eval_taylor} we can express the result as a sum of products of the associated coefficients. 

\begin{cor}[Sum-of-Products]\label{sum_of_products}
Let the Taylor series expansion of $1/f_{j}(x)$ centered about $x=a$, with $f_{j}(a)\neq 0$, be given by 
\begin{equation}\label{an_coeff}
\frac{1}{f_{j}(x)}=\sum_{n=0}^{\infty} a_{n}^{(j)}(x-a)^{n},
\end{equation}
for $j=1,\cdots,m$. Then, we have 
$$
\textnormal{eval}\left(\frac{1}{f_{1}(x)\cdots f_{m}(x)}; (x-a)^{k}\right)=\sum_{\stackrel{j=0}{j_{1}+\cdots+j_{m}=j}}^{k-1}a_{j_{1}}^{(1)}\cdots a_{j_{m}}^{(m)}(x-a)^{j}.
$$
\end{cor}

\section{Cyclotomic Polynomials and Fourier Analysis}

In this section we discuss the fruitful interaction between eval operator and the cyclotomic polynomials. The $n^{th}$ cyclotomic polynomial in $\mathbb{Q}[x]$ denoted $\Phi_{n}(x)$ is a monic irreducible polynomial whose roots are the $n^{th}$ primitive roots of unity and satisfy 
\begin{equation}\label{factor_phi}
1-x^{m}=\prod_{d|m} \Phi_{d}(x).
\end{equation}
The degree of $\Phi_{n}(x)$ is $\phi(n)$ the Euler's totient function. We also denote $\Psi_{m}(x)$ as 
$$
\Psi_{m}(x)=\frac{1-x^m}{1-x}=1+x+\cdots+x^{m-1}.
$$
In Section \ref{sec_W1W2}, we observe that the polynomial $\Psi_{m}(x)$ has an algebraic connection to the degenerate Bernoulli and degenerate Euler numbers. 

\begin{rema}
Note that many authors consider $\Phi_{1}(x)=x-1$. But, in view of factors in the generating function $F_{N}(x)$ we take $\Phi_{1}(x)=1-x$. 
\end{rema}

Let us denote the primitive $n^{th}$ roots of unity by 
$$
\Delta_{n}=\{\xi: \Phi_{n}(\xi)=0\}.
$$
The inverse cyclotomic polynomial, denoted $\Theta_{k}(x)$, is defined by Moree \cite{Moree} as 
$
1-x^{k}=\Theta_{k}(x)\Phi_{k}(x).
$
Therefore, $\Theta_{k}(\xi)\neq 0$ if $\xi\in \Delta_{k}$. In \cite{Moree}, it was shown that the coefficients of $\Theta_{k}(\xi)$ also have very good symmetric properties. 

Distinct cyclotomic polynomials are pairwise relatively prime, that is
$
\textnormal{gcd}(\Phi_{m}(x),\Phi_{n}(x))=1, \textnormal{for } m\neq n.
$
As a consequence we can express the product as 
\begin{eqnarray}
\nonumber \prod_{k=1}^{N}(1-x^{k})&=& \Phi_{1}(x)^{N}\Phi_{2}(x)^{\lfloor N/2\rfloor}\cdots\Phi_{\lfloor N/2\rfloor}(x)^{2}\Phi_{\lfloor N/2\rfloor+1}(x)\cdots \Phi_{N}(x)\\
\nonumber &=& \prod_{k=1}^{N}\Phi_{k}(x)^{\lfloor N/k\rfloor}.
\end{eqnarray}

In view of the above equation, to perform $q$-partial fraction expansion of $F_{N}(x)$ with expressions such as (\ref{hNk}) coming from the extended cover-up method we need to first obtain eval of power of cyclotomic polynomials. 

\begin{lem}[Dresden \cite{dresden}]
Let $m<n$ positive integers. Then, we have 
\begin{equation}\label{eval_mn}
\Phi_{m}(x)u(x)+\Phi_{n}(x)v(x)=1.
\end{equation}
The polynomials $u(x)$ and $v(x)$ can be obtained for the following cases:
\begin{description}
\item[Case $m\nmid n$] Let $d=\textnormal{gcd}(m,n)$ with $d=ns-mt$. Then, 
$$
u(x)=\frac{(-x)^{d}(x^{mt}-1)}{(x^{d}-1)\Phi_{m}(x)}\ \textnormal{ and }\ v(x)=\frac{x^{ns}-1}{(x^{d}-1)\Phi_{n}(x)}.
$$
\item[Case $m\mid n$] 
$$
u(x)=\frac{-\left(\Phi_{n/m}(x^{m})-\Phi_{n/m}(1)\right)}{\Phi_{n/m}(1)\Phi_{m}(x)}\ \textnormal{  and  }\  v(x)= \frac{\Phi_{n/m}(x^{m})}{\Phi_{n/m}(1)\Phi_{n}(x)}.
$$
\end{description}
Moreover, $\Phi_{n/m}(1)=p$ if $n/m=p^r$ for some prime $p$ and $r>0$. Otherwise $\Phi_{n/m}(1)=1$.
\end{lem}
From the above lemma we have 
\begin{eqnarray}
\label{alpha}\alpha(x)&=&\textnormal{eval}\left(\frac{1}{\Phi_{m}(x)}; \Phi_{n}(x)\right)=u(x)\textnormal{ rem }\Phi_{n}(x)\\
\label{beta}\beta(x)&=&\textnormal{eval}\left(\frac{1}{\Phi_{n}(x)}; \Phi_{m}(x)\right)=v(x)\textnormal{ rem }\Phi_{m}(x),
\end{eqnarray}
where $u(x),v(x)$ are given in (\ref{eval_mn}). By Lemma \ref{eval_power_k_case} we have the following corollary. 

\begin{cor}\label{eval_power_cyclo}
For $m,n$ distinct positive integers and $k$ is a positive integer we have
$$
\textnormal{eval}\left(\frac{1}{\Phi_{n}(x)};\Phi_{m}(x)^{k}\right)=\sum^{k-1}_{j=0}\tilde{a}_{j}(x)\Phi_{m}(x)^{j},
$$
where $\tilde{a}_{j}(x)=\left(\beta(x)\alpha^{(j)}(x)\right)\textnormal{ rem }\Phi_{m}(x)$ and $\alpha^{(j)}(x)=\alpha(x)^{j}\textnormal{ rem }\Phi_{n}(x)$, and $\alpha(x)$ and $\beta(x)$ are given in (\ref{alpha}) and (\ref{beta}) respectively. 
\end{cor}

Notice that in the above expressions one needs to perform modulo a cyclotomic polynomial. The following lemma along with Lemma \ref{lem_res}(2) and \ref{lem_res}(3) show that this can be done quite efficiently. 

\begin{lem} Let $m>1,k>0$ be integers. Then, we have 
$$
\textnormal{eval}\left(x^{k}; \Phi_{m}(x)\right)= x^{k\%m} \textnormal{ rem }\Phi_{m}(x).
$$
Moreover, if $m=2n$ then 
$
\textnormal{eval}\left(x^{k+n}; \Phi_{m}(x)\right)= -\textnormal{eval}\left(x^{k}; \Phi_{m}(x)\right), 
$
for $0\leq k<n.$
\end{lem}
\begin{proof}[Proof.]
Suppose $j=md+r$, for $0\leq r<m$. Then, the result follows from $\Phi_{m}(x)\mid (x^m-1)$ and 
$
x^{md+r}=((x^{m}-1)+1)^{d}x^{r}=q(x)(x^{m}-1)+x^{r}. 
$ Furthermore, it is easy to see that 
$
\textnormal{eval}(x^n+1;\Phi_{2n}(x))=(x^n+1)\textnormal{ rem }\Phi_{2n}(x)=0.
$  
\end{proof}

\begin{cor}[\cite{uk}]
Let $p$ be prime. 
$$
\textnormal{eval}(x^{j};\Phi_{p}(x))=x^{j} \textnormal{ rem } \Phi_{p}(x)= \left\{\begin{array}{@{}l@{\thinspace}l}
       x^{j\%p} & \textnormal{   if   } j\%p \neq p-1 \\
       -\sum^{p-2}_{i=0}x^{i}  & \textnormal{   if   } j\%p=p-1. \\
     \end{array}\right.
$$
\end{cor}

\begin{lem} [\cite{uk}]\label{lem_rem_xm}
For $k\geq 0$ and $m>1$, 
$
\textnormal{eval}(x^{k};1-x^{m})=x^{k\%m}.
$
\end{lem}
We now perform iterated partial fractions for the purpose of obtaining an algebraic approach to degenerate Euler numbers (discussed in Section \ref{sec_W1W2}). Similar approach for degenerate Bernoulli number is given in \cite{uk}. One can easily deduce the following partial fraction

\begin{equation}\label{pf_case_k_1}
\frac{1}{(1-x)(1+x^m)}=\frac{1}{2(1-x)}+\frac{\Psi_{m}(m)}{2(1+x^m)}.
\end{equation}
Proceeding inductively we obtain the following result. 
\begin{lem}\label{pow_k}
For $k\geq 1$ the following partial fraction holds:
\begin{equation}\label{pf_k}
\frac{1}{(1-x)^{k}(1+x^m)} = \sum_{j=0}^{k-1}\frac{g_{j}^{(m)}(1)}{2(1-x)^{k-j}}+\frac{g_{k}^{(m)}(x)}{1+x^m},
\end{equation}
where $g_{j}^{(m)}(x)=\left(\frac{1}{2}\Psi_{m}(x)\right)^{j}\textnormal{ rem } (1+x^m)$ and $g_{0}^{(m)}(x)=1$.
\end{lem}
\begin{proof}[Proof.]
We prove by induction on $k$. The equation (\ref{pf_case_k_1}) is the base case $k=1$. Suppose the result holds for $k$.  Multiplying both sides of (\ref{pf_k}) by $\frac{1}{1-x}$we get 
$$
\frac{1}{(1-x)^{k+1}(1+x^m)} =\frac{\sum_{j=0}^{k-1}g_{j}^{(m)}(1)(1-x)^{j}}{2(1-x)^{k+1}}+\frac{g_{k}^{(m)}(x)}{(1-x)(1+x^m)}.
$$
By performing partial fractions on the second term we have 
\begin{eqnarray}
\nonumber \frac{1}{(1-x)^{k+1}(1+x^m)} &=& \frac{\sum_{j=0}^{k-1}g_{j}^{(m)}(1)(1-x)^{j}}{2(1-x)^{k+1}}+\frac{g_{k}^{(m)}(1)}{2(1-x)}+\frac{g_{k+1}^{(m)}(x)}{1+x^m}\\
\nonumber &=& \frac{\sum_{j=0}^{k}g_{j}^{(m)}(1)(1-x)^{j}}{2(1-x)^{k+1}}+\frac{g_{k+1}^{(m)}(x)}{1+x^m},
\end{eqnarray}
where we used the fact $g_{k+1}(x)=\textnormal{eval}(g_{k}(x);1+x^m)$. Hence the result is proved. 
\end{proof}




\subsection{Fourier Series and $\textnormal{eval}(\cdot, \Phi_{m}(x)^r)$}\label{subsec_dft}

The power of eval with respect to cyclotomic polynomials emerges in the context of finite Fourier series. In this regard, the following substitution plays a vital role. 

\begin{lem}[Substitution Rule 2]\label{sub_2}
Let $m,k$ be positive integers and $f(x),g(x)\in \mathbb{Q}[x]$ be two polynomials such that $\textnormal{gcd}(g(x),\Phi_{m}(x))=1.$ Then, 
\begin{equation}\label{sub_2}
\left.\res{\frac{f(x)}{g(x)}}{\Phi_{m}(x)^{r}}\right|_{x=\xi}=\frac{f(\xi)}{g(\xi)}\quad \textnormal{ for }\xi\in\Delta_{m}.
\end{equation}
\end{lem}
\begin{proof}[Proof.]
Since $\textnormal{gcd}(g(x),\Phi_{m}(x)^{r})=1$ we have
$
\alpha(x)g(x)+\beta(x)\Phi_{m}(x)^{r}=1.
$
Setting $x=\xi$ in the above equation we get $\alpha(\xi)=1/g(\xi)$. Therefore,  
$$
\left.\res{\frac{f(x)}{g(x)}}{\Phi_{m}(x)^r}\right\vert_{x=\xi}= f(x)\alpha(x) \textnormal{ rem } \Phi_{m}(x)^r\vert_{x=\xi}=f(\xi)\alpha(\xi)-q(\xi)\Phi_{m}(\xi)^r
$$
for some $q(x)\in \mathbb{Q}[x]$.  
\end{proof}

For the purpose of determining the coefficients of the $q$-partial fractions we need to determine the Fourier series of 
$$
H(x)=\frac{h(x)}{\Phi_{m}(x)^{r}}=\frac{\Theta_{m}(x)^{r}h(x)}{(1-x^{m})^{r}}\quad \textnormal{ for } h(x)=\textnormal{eval}\left(\frac{f(x)}{g(x)}; \Phi_{m}(x)^{r}\right),
$$
$f(x),g(x)\in \mathbb{Q}[x]$ and $g(x)$ relatively prime to $\Phi_{m}(x)$.

It is well known that when $r=1$ the Fourier series of $H(x)$ is a finite Fourier series. So, we first determine the finite Fourier series for the case $r=1$.  
$$
H(x)=\frac{\Theta_{m}(x)h(x)}{1-x^{m}},\quad \textnormal{ for }\quad h(x)=\textnormal{eval}\left(\frac{f(x)}{g(x)};\Phi_{m}(x)\right).
$$ 
Let the sequence $(a(n))$ be a periodic function on $\mathbb{Z}$ with period $m$ and let $H(x)$ be a generating function for $(a(n))$. Then, the finite Fourier series expansion is given by 
$$
a(n)=\frac{1}{m}\sum_{\stackrel{\xi^{m}=1}{\xi\neq 1}}\Theta_{m}(\xi)h(\xi)\xi^{-n}=\frac{1}{m}\sum_{\xi\in \Delta_{m}}\Theta_{m}(\xi)h(\xi)\xi^{-n},
$$
as $\Theta_{m}(\xi)=0$ if $\Phi_{m}(\xi)\neq 0$ the sum in the above equation runs over $\Delta_{m}$. By Substitution Rule 2 (Lemma \ref{sub_2}) and the finite Fourier series expansion we deduce
\begin{eqnarray}
\nonumber H(x)&=&\frac{\Theta_{m}(x)h(x)}{1-x^{m}} =\sum^{\infty}_{n=0}a(n)x^{n}\\
\nonumber &=& \sum^{\infty}_{n=0}\left(\frac{1}{m}\sum_{\xi\in \Delta_{m}}\Theta_{m}(\xi)h(\xi)\xi^{-n}\right)x^{n} = \sum^{\infty}_{n=0}\left(\frac{1}{m}\sum_{\xi\in \Delta_{m}}\frac{\Theta_{m}(\xi)f(\xi)}{g(\xi)}\xi^{-n}\right)x^{n} .
\end{eqnarray}

Using the above reasoning we could prove that the degenerate Euler number (Equation (\ref{deg_euler}), which is obtained algebraically in Lemma \ref{pow_k}) can be expressed as a trigonometric sum. 

\begin{lem}\label{trig_deg_euler}
The polynomial $g^{(m)}_{k}(x)$ defined in Lemma \ref{pow_k} with odd $m$ satisfies the following equation for $\xi$ with $\xi^{m}+1=0$ and $\xi\neq -1$: 
\begin{equation}\label{euler_trig}
g^{(m)}_{k}(1)=\frac{2}{m}\sum_{j=0}^{m-1}\frac{\xi^{j}}{(1+\xi^{j})^{k+1}}.
\end{equation}    
\end{lem}
\begin{proof}
In order to obtain a finite Fourier series of (\ref{pf_k}), for $m$ odd we evaluate at a root of unity $\xi\neq 1$, satisfying the equation $1+\xi^{m}=0$, the function
$
g^{(m)}_{k}(x)=\textnormal{eval}(\Psi_{m}(x)^{k};1+x^{m})
$ 
we have $g^{(m)}_{k}(x)=\Psi_{m}(-\xi)^{k}$. By considering finite Fourier series of the second term of the right-hand side in (\ref{pf_k}) and the zeroth coefficient we have 
\begin{equation}\label{k_1_eqn}
1=\frac{1}{2}\sum_{j=0}^{k-1}\frac{g^{(m)}_{j}(1)}{(1-x)^{k-1-j}}+\frac{1}{m2^{k}}\sum_{j=0}^{m-1}\Psi_{m}(-\xi^{j}).
\end{equation}
By considering the $k+1$ case in (\ref{k_1_eqn}), subtracting both the equations and simplifying we obtain the required result. 
\end{proof}

Now we set ourselves for the case $r>1$. We consider
$$
H(x)=\frac{h(x)}{\Phi_{m}(x)^{r}}=\frac{\Theta_{m}(x)^{r}h(x)}{(1-x^{m})^{r}}=\frac{\tilde{h}(x)}{(1-x^{m})^{r}},
$$
where $r>1$ and $\textnormal{deg}(h)<r\phi(m)$. In this case, the associated Fourier series is no more a finite Fourier series. But one could use the tools of finite Fourier series by a simplification of the form \begin{equation}\label{decompose_h}
\tilde{h}(x)=\sum^{r-1}_{j=0}\tilde{h}_{j}(x)(1-x^{m})^{j}, \textnormal{ where } \textnormal{deg}(\tilde{h}_{j})<m,
\end{equation}
and thus we can express 
$$
H(x)=\frac{\tilde{h}(x)}{(1-x^{m})^{r}}=\sum_{j=0}^{r-1}\frac{1}{(1-x^{m})^{r-j-1}}\cdot\frac{\tilde{h}_{j}(x)}{(1-x^{m})}.
$$
Thus, in order to obtain the Fourier series for $H(x)$ we first obtain the finite Fourier series for each $\tilde{h}_{j}(x)/(1-x^{m})$ and then multiply with the series of $1/(1-x^{m})^{r-j-1}.$ An explicit expression of the form (\ref{decompose_h}) can be obtained by the following result.

\begin{thm}[\cite{uk}]\label{taylor_Dm}
Suppose $\tilde{h}(x)$ is a polynomial with $\textnormal{deg}(\tilde{h})< rm$. Then, 
$$
\tilde{h}(x)=\sum^{r-1}_{j=0}\frac{(-1)^{j}\tilde{h}^{(j)}(x)}{j!}(1-x^{m})^{j},
$$
where $\tilde{h}^{(j)}(x)=\textnormal{eval}\left(D_{m}^{j}\tilde{h}(x); 1-x^{m}\right)$ and $D_{m}(x^{k})=\left\lfloor\frac{k}{m}\right\rfloor x^{k-m}.$  In particular, $\tilde{h}^{(0)}(x)$ is the remainder of $\tilde{h}(x)$ when divided by $1-x^{m}$.
\end{thm}

Computation of $\tilde{h}^{(j)}(x)$ above can be done efficiently by Lemma \ref{lem_rem_xm}. Furthermore, from Lemma \ref{lem_rem_xm} it is also easy to see that substitution of $\xi$ an $m^{th}$ root of unity in $\tilde{h}^{(j)}(x)$ gives us 
$$
\tilde{h}^{(j)}(\xi)=\left.\textnormal{eval}\left(D_{m}^{j}\tilde{h}(x); 1-x^{m}\right)\right\rvert_{x=\xi}=D^{j}_{m}\tilde{h}(\xi).
$$
In particular, by Substitution Rule Lemma \ref{sub_2} we have  $\tilde{h}^{(0)}(\xi)=\tilde{h}(\xi)$. Reminding oneself that we have $\tilde{h}(\xi)=\Theta_{m}(\xi)h(\xi)=0$ if $\xi$ is $m^{th}$ root of unity but not in $\Delta_{m}$. Therefore, we have 
$$
\tilde{h}^{(0)}(\xi)=\begin{cases}
\Theta_{m}(\xi)h(\xi)\quad \textnormal{ if }\xi \in \Delta_{m} \\
0\quad\quad\quad\quad\quad\ \  \textnormal{if }\xi^{m}=1 \textnormal{ but } \xi \notin \Delta_{m}.
\end{cases}
$$
The Fourier series for the $j^{th}$ term corresponding to $\tilde{h}^{(j)}(x)$ can be written as 
\begin{equation}\label{gen_hij}
\frac{(-1)^{j}}{j!}\frac{1}{(1-x^{m})^{r-j-1}}\frac{\tilde{h}^{(j)}(x)}{(1-x^{m})} =\sum^{\infty}_{n=0}\sum_{n'=0}^{n}a_{1}(n')a_{2}(n-n')x^{n},
\end{equation}
where $a_{1}(n)=\frac{(-1)^{j}}{j!}\begin{pmatrix}\lfloor \frac{n}{m}\rfloor-r-j-2 \\ r+j+2\end{pmatrix}$ and $a_{2}(n)=\frac{1}{m}\sum_{\stackrel{\xi^{m}=1}{\xi\neq 1}}\tilde{h}_{j}(\xi)\xi^{-n}$. In particular, for $j=0$ we have 
\begin{equation}\label{zero_term}
\frac{\tilde{h}^{(0)}(x)}{(1-x^{m})^{r}}=\sum_{n=0}^{\infty}\sum_{n'=0}^{n} \begin{pmatrix} \lfloor \frac{n'}{m}\rfloor-r-2 \\ r+2\end{pmatrix}\left(\frac{1}{m}\sum_{\xi\in \Delta_{m}}\Theta_{m}(\xi)h(\xi)\xi^{-(n-n')}\right)x^{n}.
\end{equation}
Finally, the Fourier series for $h(x)/(1-x^{m})^{r}$ can be obtained by adding all the terms for $j=0,\cdots,(r-1)$ from (\ref{gen_hij}).

\section{The $q$-Partial Fractions of $F_{N}(x)$}

We first obtain a $q$-partial fraction of the generating function 
\begin{equation}\label{gNk}
F_{N}(x)=\prod_{k=1}^{N}\frac{1}{1-x^{k}}=\sum_{k=1}^{N}\frac{g_{k}^{(N)}(x)}{(1-x^{k})^{\lfloor N/k \rfloor}}.
\end{equation}
After a factorization of $F_{N}(x)$ into cyclotomic polynomials we have  
$$
\frac{1}{\prod_{k=1}^{N}(1-x^{k})}
=\frac{1}{\prod_{k=1}^{N}\Phi_{k}(x)^{\lfloor N/k \rfloor}}.
$$
Factors $\Phi_{k}(x)^{\lfloor N/k\rfloor}$ with distinct $k$ are pairwise relatively prime so we can apply the extended cover-up method to obtain a partial fraction to obtain 
\begin{equation}\label{pf_phi}
\prod_{k=1}^{N}\frac{1}{\Phi_{k}(x)^{\lfloor N/k \rfloor}}=\sum_{k=1}^{N}\frac{h_{k}^{(N)}(x)}{\Phi_{k}(x)^{\lfloor N/k \rfloor}},
\end{equation}
where 
$$
h_{k}^{(N)}(x)=\textnormal{eval}\left(\frac{1}{\Phi_{1}(x)^{N}\cdots \widehat{\Phi_{k}(x)^{\lfloor N/k\rfloor}} \cdots \Phi_{N}(x)};\Phi_{k}(x)^{\lfloor N/k\rfloor}\right).
$$
Upon multiplying and dividing by $\Theta_k(x)^{\lfloor N/k\rfloor}$ for each factor in (\ref{pf_phi}) and using $1-x^k=\Phi_{k}(x)\Theta_{k}(x)$ we obtain the required partial fraction (\ref{gNk}) with $g_{k}^{(N)}(x)=\Theta_k(x)^{\lfloor N/k\rfloor}h_{k}^{(N)}(x)$. One can simplify this term easily by using Corollary \ref{eval_power_cyclo}  
$$
g_{k}^{(N)}(x)=\left(\Theta_{k}(x)^{\lfloor N/k \rfloor}\sum_{\stackrel{j=0}{j_{1}+\cdots+j_{N}=j}}^{\lfloor N/k\rfloor-1}j!\frac{\tilde{a}_{j_{1}}(x)}{j_{1}!}\cdots \frac{\tilde{a}_{j_{N}}(x)}{j_{N}!}\Phi_{k}(x)^{j}\right)\textnormal{ rem } \Phi_{k}(x)^{\lfloor N/k\rfloor},
$$
where the index is always $j_{k}=0$, $\tilde{a}_{j_{k}}(x)=1$ and $\tilde{a}_{j}(x)$. Further, as $\textnormal{deg}(\tilde{a}_{j})<\phi(k)$ we can make the procedure to obtain remainder efficient.   

Now using Theorem \ref{taylor_Dm} we obtain the terms 
$$
g_{k(\lfloor N/k\rfloor-l)}^{(N)}(x)=\textnormal{eval}\left(\frac{(-1)^{l}}{l!}D_{k}^{(l)}\left(g_{k}^{(N)}(x)\right); 1-x^{k}\right), 
$$
where notice that the index $\lfloor N/k\rfloor-l$ refers to the $(l+1)^{th}$ term of the series. In particular, the first term is 
\begin{eqnarray}
\label{gNkl} g_{k\lfloor N/k\rfloor}^{(N)}(x)&=&\textnormal{eval}\left(g_{k}^{(N)}(x); 1-x^{k}\right) \\
\nonumber &=& \textnormal{eval}\left(\Theta_{k}(x)^{\lfloor N/k\rfloor}h_{k}^{(N)}(x);1-x^{k}\right).
\end{eqnarray}

\begin{thm}[$q$-Partial Fraction]\label{main_qpf}
Given $N>0$ we have the $q$-partial fraction decomposition of $F_{N}(x)$ as 
$$
F_{N}(x)=\prod_{k=1}^{N}\frac{1}{1-x^{k}}=\sum_{k=1}^{N}\sum_{l=1}^{\lfloor N/k\rfloor}\frac{g_{kl}^{(N)}(x)}{(1-x^{k})^{l}},
$$
where the term $g_{kl}^{(N)}(x)$ is in (\ref{gNkl}).
\end{thm}

\section{The Waves $W_{1}(n;N)$ and $W_{2}(n;N)$}\label{sec_W1W2} 

In this section we provide direct and explicit formulas for the second wave. In \cite{uk} the current author has demonstrated that the first wave $W_{1}(n;N)$ can be expressed in terms of the so called degenerate Bernoulli numbers. Here we show that the second wave $W_{2}(n;N)$ can be expressed in terms of degenerate Euler and degenerate Bernoulli numbers. The crucial point of our derivation is that the first and second cyclotomic polynomials are linear, i.e., $\Phi_{1}(x)=1-x$ and $\Phi_{2}(x)=1+x$.

Degenerate Bernoulli and degenerate Euler numbers have been defined by Carlitz\cite{Carlitz}. In \cite{uk} the current author demonstrated that the polynomial 
$$
\Psi_{m}(x)=\frac{1-x^m}{1-x}=1+x+\cdots+x^{m-1}
$$ 
can be used to express the degenerate Bernoulli numbers. In fact, it was shown that a Taylor series expression holds, that is
\begin{equation}\label{young_eqn}
\frac{m(1-x)}{1-x^{m}}=\frac{m}{\Psi_{m}(x)}=\sum_{k=0}^{\infty}(-1)^{k}\frac{\tilde{\beta}_{k}(m)}{k!}(1-x)^{k},
\end{equation}
where $\tilde{\beta}_{k}(m)=m^{k}\beta_{k}(1/m)$ where $\beta_{k}(m)$ is the degenerate Bernoulli number. Furthermore, if $m$ is odd and replacing $x$ by $-x$ we have 
\begin{equation}\label{bernoulli_even}
\frac{m(1+x)}{1-x^m}=\sum_{k=0}^{\infty}\frac{(-1)^{k}\tilde{\beta}_{k}(m)}{k!}(1+x)^k.
\end{equation}
Recasting the result in \cite{uk} we have the result. 
\begin{thm}[\cite{uk}]
\begin{equation}\label{W1} 
\Gamma_{01(N-j)}(N)=\frac{(-1)^{j}}{N!}\sum_{j_{2}+\cdots+j_{N}=j}\frac{\tilde{\beta}_{j_{2}}(2)\cdots \tilde{\beta}_{j_{N}}(N)}{j_{2}!\cdots j_{N}!},
\end{equation}
for $0\leq j \leq N-1$.  
\end{thm}
Using (\ref{W1}) one can easily derive the following recurrence relation
$$
\Gamma_{01(N-j)}(N+1)=\frac{1}{N+1}\sum_{k=0}^{j+1}\frac{(-1)^{k}}{k!}\tilde{\beta}_{k}(N+1)\Gamma_{01(N+k-j-1)}(N).
$$

Now, we derive a formula for the coefficients associated with the second wave. Both degenerate Bernoulli and the degenerate Euler numbers appear in coefficients associated with the second wave. Degenerate Euler number was defined by Carlitz (see \cite{kim} and the references therein) as
\begin{equation}\label{deg_euler}
\frac{2}{(1+\lambda t)^{1/\lambda}+1}=\sum_{k=0}^{\infty} \varepsilon_{k}(\lambda)\frac{t^{n}}{n!}.
\end{equation}
We consider $\lambda =1/m$, for some positive integer $m$, and substitute $1+\lambda t=x$ in (\ref{deg_euler}) to get 
$$
\frac{2}{1+x^m}=\sum_{k=0}^{\infty}\frac{(-1)^{k}\tilde{\varepsilon}_{k}(m)}{k!}(1-x)^k, \textnormal{ where } \tilde{\varepsilon}_{k}(m)=m^{k}\varepsilon_{k}(m),
$$
For $m=2n+1$, we can replace $x$ with $-x$ to obtain 
\begin{equation}\label{euler_odd}
\frac{2}{1-x^{2n+1}}=\sum_{k=0}^{\infty}\frac{(-1)^{k}\tilde{\varepsilon}_{k}(2n+1)}{k!}(1+x)^k.
\end{equation}

\begin{lem}
The degenerate Euler number $\tilde{\varepsilon}_{k}(m)$ is a polynomial in $m$.  
\end{lem}
\begin{proof}
By Lemma \ref{pow_k} we have 
$$
g_{k}(1)=(-1)^{k}\frac{\tilde{\varepsilon}_{k}(m)}{k!},
$$
where $g_{k}(x)=\left(\frac{1}{2}\Psi_{m}(x)\right)^{k} \textnormal{ rem }(x^{m}+1)$ and $\Psi_{m}(x)=1+x+\cdots+x^{m-1}$. One can easily see the result observing that $g_{k}(x)$ is a polynomial of degree $m-1$ with coefficients being polynomials in $k$ and $m$.
\end{proof}

Equipped with (\ref{bernoulli_even}) and (\ref{euler_odd}) we can obtain the necessary formula for the coefficients associated with the second wave. For the purpose of obtaining all the coefficients of the second wave we rewrite the $q$-partial fraction result as follows 

$$
\prod_{k=1}^{N}\frac{1}{1-x^{k}}=\sum_{l=1}^{\lfloor N/2\rfloor}\frac{\Gamma_{02l}(N)}{(1+x)^{l}}+\sum_{\stackrel{k=1}{k\neq 2}}^{N}\sum_{l=1}^{\lfloor N/k\rfloor}\sum_{j=0}^{k-1}\frac{\Gamma_{jkl}(N)x^{j}}{(1-x^{k})^{l}}.
$$
Once we obtain $\Gamma_{02l}(N)$ one can obtain the second wave easily using the expansion
$$
\frac{1}{(1+x)^{l}}=\sum_{n=0}^{\infty}(-1)^{n}\begin{pmatrix}t+l-1\\l-1\end{pmatrix}x^{n}.
$$
In fact, upon expansion we have
$$
W_{2}(n;N)=(-1)^{n}\sum_{l=1}^{\lfloor N/2\rfloor}\sum_{i=0}^{l-1}\begin{pmatrix}
n+l-i-1 \\ l-i-1
\end{pmatrix} \Gamma_{02l}.
$$
Therefore, we aim to obtain direct formulae for $\Gamma_{02l}(N)$. Towards this direction we have to compute 
$$
\textnormal{eval}\left(\frac{(1+x)^{\lfloor N/2\rfloor}}{(1-x)(1-x^2)\cdots (1-x^{N})}; (1+x)^{\lfloor N/2\rfloor}\right)
$$
First consider the case $N=2M$. The argument in the above expression can be written as 
$$
\frac{(1+x)^{M}}{(1-x)(1-x^2)\cdots (1-x^{2M})}=\frac{1}{1-x}\cdot \frac{1+x}{1-x^2}\ \cdots\  \frac{1}{1-x^{2M-1}}\cdot\frac{1+x}{1-x^{2M}}.
$$
Thus by determining the Taylor series expansion of the above function about $x=-1$, and using Corollary \ref{sum_of_products} along with (\ref{bernoulli_even}) and (\ref{euler_odd}) we obtain  
\begin{eqnarray}
\nonumber g(x)&=&\textnormal{eval}\left(\frac{(1+x)^{M}}{(1-x)(1-x^2)\cdots (1-x^{2M})};(1+x)^{M}\right) \\ 
\nonumber &=&  \textnormal{eval}\left(\prod_{j=1}^{M}\frac{1}{1-x^{2j-1}}\prod_{j=1}^{M}\frac{1+x}{1-x^{2j}};(1+x)^{M}\right) \\
\nonumber &=& \left(\sum^{N}_{j=0}\sum_{j_{1}+\cdots+j_{N}=j}\prod_{s=1}^{M} \frac{(-1)^{j_{2s-1}}\tilde{\varepsilon}_{j_{2s-1}}(2s-1)}{2\cdot j_{2s-1}!}\frac{(-1)^{j_{2s}}\tilde{\beta}_{j_{2s}}(2s)}{2s\cdot j_{2s}!} (1+x)^{j}\right) \textnormal{ rem }(1+x)^{M}\\
\nonumber &=& \frac{1}{2^{N}\lfloor  N/2\rfloor!}\sum_{j=0}^{M-1}(-1)^{j}\sum_{j_{1}+\cdots+j_{N}=j}^{M}\frac{\tilde{\varepsilon}_{j_{1}}(1)}{2\cdot j_{1}!}\cdots \frac{\tilde{\varepsilon}_{j_{2M-1}}(2M-1)}{2\cdot j_{2M-1}!}\cdot \frac{\tilde{\beta}_{j_{2}}(2)}{2\cdot j_{2}!}\cdots \frac{\tilde{\beta}_{j_{2M}}(2M)}{2s\cdot j_{2M}!}(1+x)^{j}.
\end{eqnarray}

\begin{thm}
For $1\leq j\leq \lfloor N/2\rfloor$
\begin{equation}\label{W2}
\Gamma_{02(\lfloor N/2\rfloor-j)}(N)=\frac{(-1)^{j}}{2^{N}\lfloor N/2\rfloor!}\sum_{j_{1}+\cdots+j_{N}=j}\frac{\tilde{\gamma}_{j_{1}}(1)\cdots\gamma_{j_{N}}(N)}{j_{1}!\cdots j_{N}!},
\end{equation}
where $\gamma_{j}(n)$ is defined as 
$$  
\gamma_{k}(n)=\begin{cases}
\tilde{\varepsilon}_{k}(n) & \textnormal{ for } n \textnormal{ odd},\\
\tilde{\beta}_{k}(n) & \textnormal{ for } n \textnormal{ even}.
\end{cases}
$$
Using (\ref{W2}) we have the following recurrence relation 
$$
\Gamma_{02(\lfloor N/2\rfloor-j)}(N+1)=\sum_{k=0}^{j}\frac{\gamma_{k}(N+1)}{k!}\tilde{\Gamma}_{02(\lfloor N/2\rfloor-j-k)}(N).
$$
\end{thm}

With the help of the formulas for the coefficients (\ref{W1}) and (\ref{W2}) we can obtain the first few terms of the $W_{1}(t;N)$ and $W_{2}(t;N)$. Towards this direction, by a direct calculation of the Taylor series we have the first three degenerate Bernoulli numbers
$$
\tilde{\beta}_{0}(m)=1,\ \tilde{\beta}_{1}(m)=\frac{m-1}{2}\ \textnormal{ and } \tilde{\beta}_{2}(m)=\frac{m^2-1}{12},
$$
and for $m$ odd the first three degenerate Euler numbers can be computed as
$$
\tilde{\varepsilon}_{0}(m)=1/2,\  \tilde{\varepsilon}_{1}(m)=m/4\ \textnormal{ and } \tilde{\varepsilon}_{2}(m)=m/8.
$$
Therefore, we have the first and second waves given below. 
\begin{eqnarray}
\nonumber W_{1}(t;N)&=&\frac{1}{N!}\begin{pmatrix}t+N-1 \\ N-1\end{pmatrix}+\frac{1}{2(N-2)!}\begin{pmatrix}t+N-2 \\ N-2\end{pmatrix}\\
\nonumber & &\quad\quad\quad\quad -\frac{(9N^2-11N-5)}{144(N-2)!}\begin{pmatrix}t+N-3 \\ N-3\end{pmatrix}+O(t^{N-4})
\end{eqnarray}
and
\begin{eqnarray}
\nonumber W_{2}(t;N)&=&\frac{(-1)^t}{2^{N}\lfloor N/2\rfloor!}\left(\begin{pmatrix}\lfloor \frac{t}{2}\rfloor+\lfloor \frac{N}{2}\rfloor-1 \\ \lfloor \frac{N}{2}\rfloor-1 \end{pmatrix}+h_{1}(N)\begin{pmatrix}\lfloor \frac{t}{2}\rfloor+\lfloor \frac{N}{2}\rfloor-2 \\ \lfloor \frac{N}{2}\rfloor-2 \end{pmatrix}\right.\\
\nonumber & & \left. \quad\quad\quad\quad\quad\quad+h_{2}(N)\begin{pmatrix}\lfloor \frac{t}{2}\rfloor+\lfloor \frac{N}{2}\rfloor-3 \\ \lfloor \frac{N}{2}\rfloor-3 \end{pmatrix}+O(t^{\lfloor N/2\rfloor-4})\right),
\end{eqnarray}
where
\begin{eqnarray}
\nonumber h_{1}(N)&=& \begin{cases} \frac{3\lfloor N/2\rfloor^2}{4} & N \textnormal{ is even} \\ \frac{3\lfloor N/2\rfloor^2+2\lfloor N/2\rfloor+2}{4} & N \textnormal{ is odd,}\end{cases} \\
\nonumber h_{2}(N)&=& \begin{cases} \frac{\lfloor N/2\rfloor^3}{18}+\frac{5\lfloor N/2\rfloor^2}{12}+\frac{\lfloor N/2\rfloor}{36} & N \textnormal{ is even} \\ \frac{\lfloor N/2\rfloor^3}{18}+\frac{5\lfloor N/2\rfloor^2}{12}+\frac{19\lfloor N/2\rfloor}{36} +\frac{1}{4} & N \textnormal{ is odd.}\end{cases} 
\end{eqnarray}

\section{The Polynomial $g^{(N)}_{k\lfloor N/k\rfloor}(x)$}

In this section, we determine the coefficients of the polynomial $g_{k\lfloor N/k\rfloor}^{(N)}(x)$, given in (\ref{qpf}),
$$
g^{(N)}_{k\lfloor N/k\rfloor}(x) = \sum_{j=0}^{k-1} \Gamma_{jk\lfloor N/k\rfloor}(N)x^{j}.
$$
As a direct consequence we can express the top-order term of the $k^{th}$ wave $W_{k}(n;N)$ 
$$
W_{k}(n;N)=\begin{pmatrix}\lfloor \frac{n}{k}\rfloor +\lfloor \frac{N}{k}\rfloor-1\\ \lfloor \frac{n}{k}\rfloor-1\end{pmatrix}\Gamma_{(n\%k)k\lfloor \frac{N}{k}\rfloor}(N) + (\textnormal{lower order terms}).
$$
For instance, for the second wave the top-order term is
$$
W_{2}(n;N)=\frac{(-1)^{n}}{2^{\lfloor N/2\rfloor+1}\lfloor N/2\rfloor!}n^{\lfloor N/2\rfloor}+ O(n^{\lfloor \frac{N}{2}\rfloor-1}).
$$
 
In order to obtain the coefficients $\Gamma_{jk\lfloor N/k\rfloor}(N)$ of $g_{k\lfloor N/k\rfloor}^{(N)}(x)$ we consider the finite Fourier series associated with  
\begin{equation}\label{gNk_series}
\frac{g_{k\lfloor N/k \rfloor}^{(N)}(x)}{1-x^{k}}=\sum_{n=0}^{\infty}a(n)x^{n}.
\end{equation} 
Because the inverse cyclotomic polynomial $\Theta_{k}(x)$ vanishes for $k^{th}$ root of unity not in $\Delta_{k}$, by evaluating $g^{(N)}_{k,\lfloor N/k\rfloor}(x)$ at $\xi$ such that $\xi^{k}=1$ and $\xi\neq 1$ we obtain
$$
g_{k,\lfloor N/k\rfloor}^{(N)}(\xi)=
\begin{cases}
\Theta_{k}(\xi)^{\lfloor N/k\rfloor}h_{k}^{(N)}(\xi)\quad \textnormal{ if }\xi\in \Delta_{k}\\
0\quad\quad\quad\quad\quad\quad\quad\quad\ \  \textnormal{ if }\xi \notin \Delta_{k}.
\end{cases}
$$
Therefore, it suffices to consider $\xi\in \Delta_{k}$. We can write the $n^{th}$ term in the series of (\ref{gNk_series})
\begin{eqnarray}
\nonumber a(n)&=&\frac{1}{k}\sum_{\xi^{k}=1, \xi\neq 1}g_{k\lfloor N/k\rfloor}^{(N)}(\xi)\xi^{-n}\\
\nonumber &=& \frac{1}{k}\sum_{\xi\in \Delta_{k}}\frac{\Theta_{k}(\xi)^{\lfloor N/k\rfloor}\xi^{-n}}{\Phi_{1}(\xi)^{N}\cdots \widehat{\Phi_{k}(\xi)^{\lfloor N/k\rfloor}} \cdots \Phi_{N}(\xi)}.
\end{eqnarray}
Consider the rational polynomial 
\begin{equation}
P_{k}^{(N)}(x)=\frac{\Phi_{1}(x)^{N}\cdots \widehat{\Phi_{k}(x)^{\lfloor N/k\rfloor}} \cdots \Phi_{N}(x)}{\Theta_{k}(x)^{\lfloor N/k\rfloor}}.
\end{equation}
By multiplying numerator and denominator with $\Phi_{k}(x)^{\lfloor N/k\rfloor}$ we get
$$
P_{k}^{(N)}(x)=\frac{\Phi_{1}(x)^{N}\cdots \widehat{\Phi_{k}(x)^{\lfloor N/k\rfloor}} \cdots \Phi_{N}(x)}{\Theta_{k}(x)^{\lfloor N/k\rfloor}}=\frac{(1-x)\cdots (1-x^{N})}{(1-x^{k})^{\lfloor N/k\rfloor}}.
$$
$P_{k}^{(N)}(x)$ is a polynomial. Indeed, by rearranging factors we have 
\begin{eqnarray}
\nonumber P_{k}^{(N)}(x)&=& \prod_{s=0}^{\lfloor N/k\rfloor-1}\prod_{j=1}^{k-1}(1-x^{sk+j})\prod_{i=1}^{\lfloor N/k\rfloor}\frac{1-x^{ik}}{1-x^{k}}\prod_{j=k\lfloor N/k\rfloor+1}^{N}(1-x^{j})\\
\label{PNkx} &=& \prod_{s=0}^{\lfloor N/k\rfloor}\prod_{j=1}^{k-1}(1-x^{sk+j})\prod_{i=1}^{\lfloor N/k\rfloor}\sum_{j=0}^{i-1}x^{kj}\prod_{j=k\lfloor N/k\rfloor+1}^{N}(1-x^{j}).
\end{eqnarray}

\begin{thm}

For $k$ an integer such that $1\leq k\leq N$,
\begin{equation}\label{gamma_jknk}
\Gamma_{j,k,\lfloor N/k\rfloor}(N)=\frac{1}{k}\sum_{\xi \in \Delta_{k}}\frac{\xi^{-j}}{P_{k}^{(N)}(\xi)}
\end{equation}
where $P_{k}^{(N)}(x)$ is given in (\ref{PNkx}).
\end{thm}

In order to further simplify (\ref{gamma_jknk}) we first define a special trigonometric sum. 

\begin{defn}\label{GRS} 
Given $0\leq s<k$ two integers, we define the Gaussian-Ramanujan sum as 
$$
\sigma_{k}(t; s)=\sum_{\xi \in \Delta_{k}}\frac{\xi^{-t}}{(\xi)_{s}},
$$
where $\Delta_{k}$ is the set of all $k^{th}$ primitive roots of unity,  and $(q)_{n}$ is given in the $q$-Pochhammer notation 
$
(q)_{n}=(q;q)_{n}=\prod_{i=1}^{n}(1-q^{i}),$ for $n\geq 1$ and $(q)_{0}=1.$
\end{defn} 

It can be easily seen that $\sigma_{k}(t;s)$ is periodic in $t$ with a periodicity $k$. The values of $\sigma_{k}(t;s)$ are of the form $l/k$, where $l$ is an integer. We call the sum $\sigma_{k}(t;s)$ the Gaussian-Ramanujan sum as it can be expressed as a linear combination of Ramanujan sums with coefficients associated with the Gaussian binomial coefficients evaluated at the primitive roots of unity. In particular, we have $
\sigma_{k}(t;0)=c_{k}(-t)=c_{k}(t), 
$ and $\sigma_{k}(t;k-1)=\frac{1}{k}\sigma_{k}(t;0)=\frac{1}{k}c_{k}(t)$, where $c_{k}(t)$ is the Ramanujan sum.

Substituting $\xi\in \Delta_{k}$ for $x$ in $P_{k}^{(N)}(x)$ given in (\ref{PNkx}) and using 
\begin{equation}\label{k_eqn}
\prod_{j=1}^{k-1}(1-\xi^{j})=k\quad \textnormal{ with }\quad \sum_{j=0}^{i-1}\xi^{kj}=i
\end{equation}
we have 
$$
P_{k}^{(N)}(\xi)=k^{\lfloor N/k\rfloor}\lfloor N/k\rfloor!(1-\xi)\cdots (1-\xi^{N\%k}).
$$
Therefore, by (\ref{gamma_jknk}) we have 
\begin{eqnarray}
\nonumber \Gamma_{jk\lfloor N/k\rfloor}(N)&=&\frac{1}{k}\sum_{\xi \in \Delta_{k}}\frac{\xi^{-j}}{P_{k}^{(N)}(\xi)} \\
\nonumber &=& \frac{1}{k^{\lfloor N/k\rfloor+1}\lfloor N/k\rfloor!}\sum_{\xi\in \Delta_{k}}\frac{\xi^{-j}}{(1-\xi)\cdots (1-\xi^{N\%k})} \\
\nonumber &=& \frac{1}{k^{\lfloor N/k\rfloor+1}\lfloor N/k\rfloor!}\sigma_{k}(j;N\%k).
\end{eqnarray}
Let $R_{k,s}(x)=\sum^{k-1}_{t=0}\sigma_{k}(t;s)x^{t}$. This proves Theorem \ref{main_gamma_jkNk}. As a consequence of this result, one can observe that it suffices to compute only $k$ polynomials in order to determine all $g_{k\lfloor N/k\rfloor}^{(N)}(x)$ for all $k=1,2,\dots,N$. 

\begin{exmp} For $k=3$ we have 
\begin{eqnarray}
\nonumber g_{3,\lfloor N/3\rfloor}^{(N)}(x) &=& \Gamma_{0,3,\lfloor N/3\rfloor}(N)+\Gamma_{1,3,\lfloor N/3\rfloor}(N)x+\Gamma_{2,3,\lfloor N/3\rfloor}(N)x^2 \\
\nonumber &=& \begin{cases}
\frac{1}{3^{\lfloor N/3\rfloor+1}\lfloor N/3\rfloor!}(-x^2-x+2), & \text{if }N\%3=0\\
\frac{1}{3^{\lfloor N/3\rfloor+1}\lfloor N/3\rfloor!}(1-x^{2}), & \text{if }N\%3=1\\
\frac{1}{3^{\lfloor N/3\rfloor+2}\lfloor N/3\rfloor!}(-x^2-x+2) & \text{if }N\%3=2.
		 \end{cases}
\end{eqnarray}
\end{exmp}

\begin{cor} The below results can be directly proved from the properties of Ramanujan sums.
\begin{enumerate}
    \item $g_{k\lfloor N/k\rfloor}^{(N)}(x)=\frac{1}{k^{\lfloor N/k\rfloor}\lfloor N/k\rfloor!}g_{k\lfloor N/k\rfloor}^{(N\% k)}(x)$. 
    
\item Suppose $N=kM$, we have 
$$
g^{(N+k-1)}_{k,M}(x)=\frac{1}{k}g^{(N)}_{k,M}(x).
$$ 

\item If $k=p^{l}$,where $p$ is a prime, $l$ is an integer and let $N=kM$, then $$g^{(N)}_{p^{l},M}(x)=\frac{p^{l-1}}{p^{l(M+1)}M!}\left(p-\Psi_{p}(x^{p^{l-1}})\right).$$

\item  If $k=pq$ for distinct primes $p,q$ and let $N=kM$, then $$g_{pq,M}^{(N)}(x)=\frac{1}{(pq)^{M+1}M!}\left(pq+\Psi_{pq}(x)-p\Psi_{q}(x^{p})-q\Psi_{p}(x^{q})\right).$$

\end{enumerate}

\end{cor}






For a given $k$ one can obtain an $O(k^2)$ time algorithm to compute all $k^2$ values $\sigma_{k}(t;j)$ for $0\leq j < k$ and $0\leq t<k$ using the below recurrence relation. Moreover, using the von Sterneck's arithmetic function one can compute the Ramanujan sum efficiently. 

\begin{lem}[Recurrence Relation]\label{pascal}
For $0<s<k$ we have 
$$
\sigma_{k}(t;s)=\sigma_{k}(t;s-1)+\sigma_{k}(t-s;s),
$$
with $\sigma_{k}(t;0)=c_{k}(t)$, the Ramanujan sum. 
\end{lem}
\begin{proof}
The result is straightforward by $(\xi)_{s}=(\xi)_{s-1}-\xi^{s}(\xi)_{s-1}$ for $s\geq 1$.
\end{proof}
By Theorem \ref{main_gamma_jkNk}, we have the coefficient $\Gamma_{jk\lfloor N/k\rfloor}$ being just a scaling factor of $\sigma_{k}(j,N\%k)$ it is not surprising (or perhaps, surprising) that the recurrence relation is identical to the one satisfied by the restricted partition function
$$ 
p_{s}(t)=p_{s-1}(t)+p_{s}(t-s).
$$ 
From the above recurrence relation we have 
$$
R_{k,s-1}(x)=(1-x^{k-
s})R_{k,s}(x) \mod 1-x^k,
$$
and 
$$
R_{k,k-1}(x)=\frac{1}{k}(x)_{k-1}\mod (1-x^k)=\frac{1}{k}\sum_{t=0}^{k-1}c_{k}(t)x^{t},
$$
where $c_{k}(t)$ is the Ramanujan sum. Using the discrete Fourier transform we can prove the following result. 
\begin{lem}\label{R_k_s_poly}
$$
R_{k,s}(x)=\frac{1}{k}(x)_{k-1}(x)_{k-1-s} \mod (1-x^{k}).
$$
\end{lem}
\begin{proof}
Let $\xi\in \Delta_{k}$ and $R_{k,s}(x)=\frac{1}{k}(x)_{k-1}(x)_{k-1-s}\mod (1-x^{k})=\sum_{\alpha=0}^{k-1}a_{t}x^{t}$. Then 
$$
R_{k,s}(\xi^{j})=\sum_{a=0}^{k-1}a_{t}\xi^{jt}=\frac{1}{k}(\xi^{j})_{k-1}(\xi^{j})_{k-1-s}.
$$
We solve this by taking an inverse discrete Fourier transform to obtain 
\begin{equation}\label{a_t}
a_{t}=\frac{1}{k}\sum_{j=0}^{k-1}R_{k,s}(\xi^{j})\xi^{-jt}.
\end{equation}

$$
R_{k,s}(\xi^{j})=\begin{cases} 0 & \textnormal{gcd}(j,k)\neq 1 \\ (\xi^{j})_{k-1-s} = k/(\xi^{j})_{s}& \textnormal{gcd}(j,k)= 1. \end{cases}
$$
Therefore, we have 
$$
a_{t}=\frac{1}{k}\mathop{\sum_{0\leq j\leq k}}_{\textnormal{gcd}(j,k)=1}\xi^{-jt}\frac{k}{(\xi^{j})}_{s}=\sum_{\eta\in \Delta_{k}}\eta^{-t}/(\eta)_{s}=\sigma_{k}(t;s).
$$
\end{proof}

\begin{cor} Let $s< k/p$, where $p$ is the smallest prime divisor of $k$ ($p=1$ if $k$ is a prime). Then,  
$$
R_{k,s}(x)=(x)_{k-1-s} \mod (1-x^{k}).
$$
\end{cor}
\begin{proof}
From the condition $s<k/p$ we can see that the set $\{k-s,\cdots,k-1\}$ does not contain any divisor of $k$. Thus, the factors $(1-x^{k-s-1}),\cdots, (1-x^{k-1})$ do not vanish for any $x=\xi^{j}$, $\xi\in \Delta_{k}$ and $j=1,2,\cdots,(k-1)$. Therefore, we have $R_{k,s}(\xi^{j})=(\xi^{j})_{k-1-s}$ in (\ref{a_t}).       
\end{proof}

\begin{cor} For $p$ prime
$$
R_{p,s}(x)=\sum_{t=0}^{p-1}\sigma_{p}(t,s)x^{t}=(x)_{s} \mod (1-x^{p}),\quad \textnormal{ for } p \textnormal{ prime and }s\geq 1.
$$
\end{cor}

By the parity counting partition and sieving of coefficients, for an arbitrary $s\in \{0,1,\cdots,k-1\}$, we have the following interpretation:
$$
\sigma_{k}(t;k-1-s)=\frac{1}{k} \sum_{j=0}^{\lfloor\frac{k-1}{2}\rfloor}\left(\tilde{E}_{s,k}(jk+t)-\tilde{O}_{s,k}(jk+t)\right), 
$$
where $\tilde{E}_{k,s}(l)$ (respectively $\tilde{O}_{k,s}(l)$) is the number of partitions of $l$ into an even (respectively an odd) number of parts each less than $k$ and permitting parts of size $\leq s$ at most twice and others at most once. 

If $s< k/p$, where $p$ is the smallest prime factor of $k$ ($p=1$ if $k$ is a prime), then the interpretation is simpler:
\begin{equation}\label{parity_count}
\sigma_{k}(t;k-1-s)=\sum_{j=0}^{\lfloor\frac{k-1}{2}\rfloor}\left(E_{s}(jk+t)-O_{s}(jk+t)\right)
\end{equation}
where $E_{s}(l)$ (respectively $O_{s}(l)$) is the number of partitions of $l$ into an even (respectively an odd) number of distinct parts each less than $k$. The result (\ref{parity_count}) is known for the case $s=k-1$, i.e., for the Ramanujan sum \cite{ramanathan}.

In order to perform further analysis, we write the Gaussian Ramanujan sum as given below. 

\begin{lem} 
$$
\sigma_{k}(t;s)=\frac{1}{k}
\sum_{\xi \in \Delta_{k}}\xi^{t}(\xi)_{k-1-s}.
$$
\end{lem}
\begin{proof}
For $\xi\in \Delta_{k}$, we know  $(1-\xi)\cdots (1-\xi^{k-1})=k.$ Multiplying and dividing by $(1-\xi^{s+1})\cdots (1-\xi^{k-1})$ in the summation we have 
\begin{eqnarray}
\nonumber \sigma_{k}(t;s)&=& \frac{1}{k} \sum_{\xi\in\Delta_{k}}\xi^{-t}(1-\xi^{s+1})\cdots (1-\xi^{k-1})\\
\nonumber &=& k\sum_{\eta\in \Delta_{k}}\eta^{t}(1-\eta^{k-s-1})\cdots (1-\eta),\quad \textnormal{ where } \eta=\xi^{-1} \textnormal{ and } \xi^{-k}=1.
\end{eqnarray}
\end{proof}

Using the above formula we can write
\begin{eqnarray}
\nonumber \eta_{k}(t,s)&=& \sum_{d|k}\sigma_{d}(t;s) = \frac{1}{k}\sum_{d|k}\sum_{\xi\in \Delta_{d}}\xi^{t}(\xi^{j})_{k-1-s} = \frac{1}{k}\sum_{j=0}^{k-1}\eta^{jt}(\eta^{j})_{k-1-s},
\end{eqnarray}
where $\eta$ is a primitive $k^{th}$ root of unity. From the linear recurrence relation we have 
$$
\eta_{k}(t,s)=\eta_{k}(t,s-1)+\eta_{k}(t-s,s).
$$
Also, notice that $\eta_{d}(t;s)=0$ for $d|k$ and $d+s<k$. By the m\"{o}bius inversion we have 
$$
\sigma_{k}(t;s)=\sum_{d|k}\mu\left(\frac{k}{d}\right)\eta_{d}(t;s).
$$
Using the above equation we can express the Gaussian-Ramanujan sum for simple cases. 
\begin{exmp}
For $k>2$ and $1\leq t\leq k$ we have
$$
\sigma_{k}(t;k-2)=\frac{1}{k}\sum_{\delta=0}^{1}\sum_{d|(k,t-\delta)}(-1)^{\textnormal{wt}(\delta)}\mu\left(\frac{k}{d}\right)d.
$$
For $3\leq t\leq k+2$ we have
$$
\sigma_{k}(t;k-3)=\frac{1}{k}\sum_{\delta=0}^{3}\sum_{d|(k,t-\delta)}(-1)^{\textnormal{wt}(\delta)}\mu\left(\frac{k}{d}\right)d, 
$$
where $\textnormal{wt}(\delta)$ is the number of $1$s in the binary representation of $\delta$. 
\end{exmp}











In view that $\xi$ is a $k^{th}$ primitive root of unity, we obtain the sieved sum
$$
\xi^{t}(\xi)_{s}=\sum^{s(s+1)/2}_{i=0}a_{i,s}\xi^{i+t}=\sum_{j=0}^{k-1}M_{t,j,s}^{(k)}\xi^{j},
$$
where 
$
M_{t,j,s}^{(k)}=\sum_{\alpha=0}^{\lfloor s(s+1)/2k\rfloor}a_{j-t+\alpha k,s}
$. Thus, we have 
$$
\sigma_{k}(t;k-1-s)=\sum_{\xi\in \Delta_{k}}\xi^{-t}(\xi)_{s}=\sum_{j=0}^{k-1}M_{t,j,s}^{(k)}c_{k}(j).
$$
Explicit formulas for $M_{t,j,s}^{(k)}$ for certain special cases were given in \cite{Goswami}. We can obtain a bound as follows 
$$
|\sigma_{k}(t;k-1-s)|\leq \varphi(k)\sum_{j=0}^{k-1}|M_{t,j,s}^{(k)}|.
$$
We provide bounds for some special cases using the bounds on the Ramanujan sum.
\begin{thm} The following bounds hold:
\begin{enumerate}
    \item $|\sigma_{k}(t;0)|\leq \varphi(k).$
    \item $|\sigma_{k}(t;k-1)|\leq \frac{1}{k}\varphi(k).$
    \item $|\sigma_{k}(t;k-2)|\leq \frac{2}{k}\varphi(k).$
    \item $|\sigma_{k}(t;1)|\leq (k-1)\varphi(k)/2.$
\end{enumerate}
\end{thm}
\begin{proof}
The results (1),(2) and (3) can be easily proved using the linear recurrence relation given in Lemma \ref{pascal} and the bound $|c_{k}(t)|\leq \varphi(k)$. In order to prove (4), we consider $\xi\in \Delta_{k}$. By $(\xi)_{k-1}=k$ and $(1-\xi)(\xi+2\xi^{2}+\cdots+(k-1)\xi^{k-1})=-k$ we have 
\begin{eqnarray}
\nonumber \sigma_{k}(t;1)&=&\sum_{\xi\in \Delta_{k}}\frac{\xi^{-t}}{1-\xi}= \frac{-1}{k}\sum_{\xi\in \Delta_{k}}(\xi^{-t}+2\xi^{2-t}+\cdots+(k-1)\xi^{k-1-t})\\
\nonumber &=& \frac{-1}{k}\sum_{j=1}^{k-1}jc_{k}(j-t).
\end{eqnarray}
Therefore, we have the bound 
$
|\sigma_{k}(t;1)|\leq (k-1)\varphi(k)/2.
$
\end{proof}

For large values of $s$ one can directly use the bounds given by Sudler \cite{sudler}.

\begin{thm}
 For $k-\sqrt{2k}-1<s$ we have 
 $$
 \log \max_{0\leq t<k} |\sigma_{k}(t;s)| = Ks+O(\log s),
 $$
 where $K<0.2$, a constant.  
\end{thm}

\begin{proof}
 We have $(x)_{k-1-s}\mod (1-x^{k})=(x)_{k-1-s}$ if the degree of $(x)_{k-1-s}$ is less than $k$. That is, 
 $(k-1-s)(k-s)/2<k.$ Implies, $(k-s-1)^{2}<2k$. So, there is no sieving of the sum. Hence, the bounds on $\sigma_{k}(t;s)$ is just given by the bounds on the coefficients of $(x)_{s}$ given in \cite{sudler}. 
\end{proof}

\end{document}